\documentclass[12pt]{amsart}

\usepackage[T1]{fontenc}
\usepackage{mathpazo}
\usepackage[none]{hyphenat}
\usepackage{mathpple}
\usepackage{amsmath,amsfonts}
\usepackage{amsthm}
\usepackage[all]{xy}
\usepackage{color}
\usepackage{tikz}

\usepackage{hyperref}

\numberwithin{equation}{section}

\newcommand{\on}{\operatorname}

\newcommand{\Z}{\mathbb{Z}}

\newcommand{\nc}{\newcommand}
\nc{\mc}{\mathcal}

\nc{\D}{\mathbb{D}}

\nc{\h}{\mathfrak{h}}
\nc{\g}{\mathfrak{g}}
\nc{\n}{\mathfrak{n}}
\nc{\ch}{\on{CH}}
\nc{\wt}{\widetilde}

\nc{\G}{\mc{G}}
\nc{\ov}{\overline}

\nc{\AutO}{\on{Aut}{\mc{O}}}
\nc{\AutmO}{\on{Aut}^{(m)} {\mc{O}}}

\nc{\VFun}{\on{Vect}(\fun)}

\nc{\FF}{\mathbb{F}}

\nc{\Amod}{$A$-{mod}}
\nc{\C}{\mathbb{C}}

\nc{\Pone}{\mathbb{P}^1}
\nc{\Aone}{\mathbb{A}^1}
\nc{\mt}{\widetilde{M}}
\nc{\Ak}{\mathbb{A}^k}
\nc{\xim}{\xi_m}
\renewcommand{\ss}{\widetilde{g}}

\nc{\J}{\mathcal{J}}
\nc{\Jtw}{\mathcal{J}^{g}}
\nc{\JX}{\mathcal{J}X}
\nc{\JtwX}{\mathcal{J}^{g}X}
\nc{\V}{V}
\nc{\vac}{{\bf 1}}
\nc{\End}{\on{End}}

\nc{\DerO}{\on{Der}_o \mc{O}}

\nc{\spec}{\on{Spec}}

\nc{\coinv}{H_{G,Z}(Y,\wt{y}_1, \cdots, \wt{y}_s)}
\nc{\cblocks}{C_{G,Z}(Y,\wt{y}_1, \cdots, \wt{y}_s)}

\nc{\ps}{\mathbb{C}[[z]]}
\nc{\psm}{\mathbb{C}[[z^{1/m}]]}

\theoremstyle{plain}
\newtheorem{theorem}{Theorem}[section]
\newtheorem{thm-defn}{Theorem/Definition}[section]

\newtheorem{lem-defn}{Lemma/Definition}[section]

\theoremstyle{definition}
\newtheorem{rmk}{Remark}[section]

\newtheorem{example}{Example}
\newtheorem{definition}{Definition}
\newtheorem*{theorem*}{Theorem}
\newtheorem*{corollary*}{Corollary}

\theoremstyle{remark}

\numberwithin{example}{section}
\numberwithin{definition}{section}

\allowdisplaybreaks[4]  

\begin{document}

 \title{Twisted modules and co-invariants for commutative vertex algebras of jet schemes }
  \author{ Matt Szczesny}
  \date{}

  \maketitle

\begin{abstract}
Let $Z \subset \mathbb{A}^k$ be an affine scheme over $\C$ and $\J Z$ its jet scheme. It is well-known that $\mathbb{C}[\J Z]$, the coordinate ring of $\J Z$, has the structure of a commutative vertex algebra. This paper develops the orbifold theory for $\mathbb{C}[\J Z]$.  A finite-order linear automorphism $g$ of $Z$ acts by vertex algebra automorphisms on $\mathbb{C}[\J Z]$.  We show that $\mathbb{C}[\J^g Z]$, where $\J^g Z$ is the scheme of $g$--twisted jets has the structure of a $g$-twisted $\mathbb{C}[\J Z]$ module. We consider spaces of orbifold coinvariants valued in the modules $\mathbb{C}[\J^g Z]$ on orbicurves $[Y/G]$, with $Y$ a smooth projective curve and $G$ a finite group, and show that these are isomorphic to $\mathbb{C}[Z^G]$. 
\end{abstract}

\section{Introduction}


Let $Z \subset \Ak$ be an affine scheme over $\C$, and $$\J Z := \on{Hom}_{Sch}(\spec \mathbb{C}[[t]], Z)$$ its jet scheme. It is well-known \cite{FBZ, BD} that the coordinate ring $\C[\J Z]$ has the structure of a commutative vertex algebra. Such vertex algebras often arise as quasiclassical limits of noncommutative vertex algebras, and have found a number of applications,  such as in the study of chiral differential operators and the invariant theory of vertex algebras \cite{A, AM, LSS}. This paper is devoted to the orbifold theory of the commutative vertex algebra $\C[\J Z]$, or more specifically, to the construction of twisted modules for $\C[\J Z]$ and coinvariants valued in such. 

Given a linear automorphism $g: Z \rightarrow Z$ of finite order $m$, we obtain an induced action on $\J Z$ and hence on $\C[\J Z]$ by vertex algebra automorphisms. We may also associate to this data the \emph{$g$--twisted jet scheme}
\[
\J^g Z := \{ x(t^{1/m}) \in \on{Hom}(\spec \C[[t^{1/m}]], Z)  \vert x(  e^{2 \pi i / m} t^{1/m}) = g(x(t^{1/m}) ) \}
\]
of g-equivariant jets. 
An abbreviated version of our result is the following : 

\begin{theorem*}[\ref{twisted_theorem}]
$\C[\J^g Z]$ carries the structure of $g$--twisted $\C[\J Z]$--module. 
\end{theorem*}

Suppose now that $Y$ is a smooth projective curve with an effective action of the group $G$. We proceed to study the space of coinvariants for the vertex algebra $\C[\J Z]$ on the orbicurve (or stacky curve) $[Y/G]$. We follow the approach of \cite{FSz}, which entails defining coordinate-independent versions of twisted vertex operators as sections of an certain sheaf on $[Y/G]$. More precisely, we use the $g$--twisted module structure on $\C[\J^g Z]$ to produce an equivariant section $\mc{Y}_y$ near the point $[y/ \langle g \rangle]$. Using the sections $\mc{Y}_y$, we define a space of coinvariants $\coinv$ for the vertex algebra $\C[\J Z]$ valued in the twisted modules $\C[\J^g Z]$. Our result is as follows:

\begin{theorem*}[\ref{coinv_theorem}]
$\coinv$ is isomorphic to $\C[Z^G]$ - the coordinate ring of the fixed-point set of $G$ on $Z$,
\end{theorem*}

\noindent When $G$ is the trivial group, and there are no twisted module insertions, the space of coinvariants is simply $\mathbb{C}[Z]$, which recovers a result proven in section 9.4.4 of \cite{FBZ}. 

The outline of the paper is as follows. In section \ref{va} we recall some basics on vertex algebras and their twisted modules. Section \ref{jet_schemes} reviews the construction of jet schemes. In section \ref{twisted_mods} we prove Theorem \ref{twisted_theorem}. Finally, in section \ref{coinvariants} we recall the coordinate-independent construction or orbifold coinvariants from \cite{FSz}, and prove Theorem \ref{coinv_theorem}. 

\medskip
\noindent {\bf Acknowledgments:} The author gratefully acknowledges the support of a Simons Foundation Collaboration Grant, as well the hospitality of the Perimeter Institute where part of this work was completed.

\section{Vertex algebras and twisted modules} \label{va}

In this section, we recall some basic definitions regarding vertex algebras and their twisted modules. We refer the reader to \cite{FBZ,K} for further information regarding vertex algebras. 

\begin{definition}
A {\em vertex algebra} is a vector space $\V$ equipped with:
\begin{itemize}
\item a linear map
\begin{align*}
Y: \V & \rightarrow End(\V)[[z,z^{-1}]] \\
     a & \rightarrow Y(a,z) = \sum_{n \in \mathbb{Z}} a_{(n)} z^{-n-1}
\end{align*}
\item a vector $\vac \in \V$, called the {\em vacuum vector},
\item a linear operator $T: \V \rightarrow \V$,
called the {\em translation operator}.
\end{itemize}
which are required to satisfy the following properties:
\begin{enumerate}
 \item $Y(Ta, z)=\partial_z Y(a,z)$,
\item $Y(\vac,z)=id_{\V}$,
\item $Y(a,z)\vac \in \V[[z]]$ and 
$a_{(-1)}\vac=a$,
\item For $m,n, k \in \mathbb{Z}$
\begin{align*}
& \sum_{j\geq 0}\begin{pmatrix}
		m\\j
	       \end{pmatrix}(a_{(n+j)}b)_{(m+k-j)}\\
&=
\sum_{j\geq 0}(-1)^j \begin{pmatrix}
		      n\\ j
		     \end{pmatrix}
\left(a_{(m+n-j)}b_{(k+j)}-(-1)^n b_{(n+k-j)}a_{(m+j)}\right).
\end{align*}
\end{enumerate}
\end{definition}

\begin{example}[Commutative Vertex Algebras] \label{comm_va}
Let $A$ be a commutative algebra over $\C$ equipped with a derivation $T_A$. We may give $A$ the structure of vertex algebra by taking $\V = A$, $T=T_A$, $\vac=1_A$, and defining
\[
Y(a,z) = e^{zT}(a) = \sum_{n \geq 0} \frac{z^n}{n !} T^n (a) 
\]
Conversely, given any vertex algebra $V$ such that $Y(a,z) \in V[[z]], \forall a \in V $, the operation $ab := Y(a,z) b \vert_{z=0}$ makes $V$ into a commutative algebra with derivation $T$. We note that for commutative vertex algebras, $Y$ is multiplicative, i..e $$ Y(ab,z) = Y(a,z) Y(b,z) .$$ \qedsymbol
\end{example}

A { \em vertex algebra automorphism} consists of a linear map $g: V \rightarrow V$ such that $$ Y(g(a),z) = g Y (a,z) g^{-1} \; \forall a \in V,$$ or equivalently, such that $$g(a)_{(n)} g(b) = g(a_{(n)} b) \;  \forall a,b \in V.$$ For an automorphism $g$ of $V$ of finite order $m$,
set $$V^r=\{u\in V\ |\ gu=\zeta_m^{r}u\}, \;  0\leq r\leq m-1, $$ where $\zeta_m = \exp{(2 \pi i/m)}$.
We recall the definition of $g$-twisted $V$-modules:

\begin{definition}\label{def:weak-twisted}
Let $g$ be an automorphism of $V$ of order $m$. A $g$-twisted $V$-module $M$ is a vector space equipped with a
linear map
\begin{align*}
Y_M: \V & \rightarrow End(M)[[z^{1/m},z^{-1/m}]] \\
     a & \rightarrow Y_M (a,z) = \sum_{n \in \frac{1}{m}\mathbb{Z}} a_{(n)} z^{-n-1}
\end{align*}
which satisfies the following conditions:

\begin{enumerate}
\item $Y_M(a,z^{1/m}) = \sum_{i \in r/p+\Z}u_i z^{-i-1}$ for $a \in V^r$.
\item $Y_M(a,z^{1/m})v\in M((z^{1/m}))$ for $a \in V$ and $v \in M$.
\item $Y_M({\mathbf 1},z^{1/m}) = \mathrm{id}_M$.
\item For $a \in V^r$, $b \in V^{s}$, 
$m\in r/T+\Z,\ n\in s/T+\Z$, and $l\in\Z$,
\begin{align*}
&\sum_{i=0}^{\infty}\binom{m}{i}
(a_{l+i}b)_{m+n-i} \\
& = 
\sum_{i=0}^{\infty}\binom{l}{i}(-1)^i
\big(a_{l+m-i}b_{n+i}+(-1)^{l+1}b_{l+n-i}a_{m+i}\big).
\end{align*}
\end{enumerate}
\end{definition}

We note for future reference that the property 
\begin{equation} \label{twisted_derivative}
Y_M (Ta,z^{1/m}) = \partial_z Y_M(a,z^{1/m})
\end{equation}
holds in any twisted modulle $M$. 

\begin{rmk} \label{multiplicativity}
It follows from property $(4)$ above by taking $l=-1$ that if $V$ is a commutative vertex algebra as in Example \ref{comm_va}, then $Y_M$ is multiplicative, i.e.
$$ Y_M (ab,z^{1/m}) = Y_M (a,z^{1/m})Y(b,z^{1/m}).$$
\end{rmk}

\section{Jet schemes} \label{jet_schemes}

Let $Z \subset \Ak$ be an affine scheme. We write $Z = \spec (A)$, where
\[
A = \C[x_1, \cdots, x_k] / (P_1, \cdots, P_r)
\]
for some polynomials $P_1, \cdots, P_r \in \C[x_1, \cdots, x_k]$. Recall that the {\em jet scheme} of $Z$ is the scheme $\J Z$ defined by the property 
\[
\on{Hom}_{Sch}(\spec R, \J Z) = \on{Hom}_{Sch} (\spec R[[t]], Z)
\]
for any commutative $\C$--algebra $R$. $\J Z$ therefore represents the space of maps from the formal disk $D=\spec \C[[t]]$ to $Z$. Writing a map $D \rightarrow \Ak$ as
\[
x_{i}(t) = \sum_{n \leq 0} x_{i,n} t^{-n}, \; 1 \leq i \leq k,  
\]
$\J Z$ may be explicitly described as $\spec A_{\infty}$, where 
\begin{equation} \label{funct_jet_sch}
A_{\infty} = \C [x_{1,n}, x_{2,n}, \cdots, x_{k,n}]_{n \leq 0} / (P_{1,n}, \cdots, P_{r,n})
\end{equation}
and 
\begin{equation} \label{Pin}
P_{i,n}= \frac{\partial^n_t}{n!} P_i (x_{1}(t), x_2 (t), \cdots, x_k (t)) \vert_{t=0} 
\end{equation}

Identifying the variables $x_i$ with $x_{i,0}$ we obtain a $\C$--algebra homomorphism $A \rightarrow A_{\infty}$ which is dual to the canonical projection \[ \mu: \J Z \rightarrow Z \] that evaluates a jet at $t=0$. 

Suppose now that $g: \Ak \rightarrow \Ak$ is a linear automorphism of order $m$. After a linear change of coordinates we may diagonalize $g$ such that its action is given by 
\begin{equation} \label{sigma_action}
 g(x_i) := \xim^{\alpha_i} x_i 
\end{equation}
with $\xim=\exp(2 \pi i / m)$. Let $$\Jtw \Ak = \{ x(t^{1/m}) \in \on{Hom}(\spec \C[[t^{1/m}]], \Ak)  \vert x( \xim t^{1/m}) = g(x(t^{1/m}) ) \}. $$ We refer to $\J^g \Ak$ as the scheme of $g$--twisted jets to $\Ak$. It is the closed subscheme of $\J \Ak = \on{Hom}(\spec \C[[t^{1/m}]], \Ak) $ consisting of fixed points under the action of $g$ by $g(x(t^{1/m})) :=  g^{-1} x( \xim t^{1/m})$. Writing $$x_i (t) = \sum_{n \in \frac{1}{m }\mathbb{Z}_{\leq 0}} x_{i,n}t^{-n}$$ we see that 
$$ \Jtw \Ak = \spec \C[x_{i,n}], \; \; i=1 \cdots k, n \in \frac{\alpha_i}{m}+\mathbb{Z}, n \leq 0. $$ 

Suppose furthermore that the action of $g$ on $\Ak$ preserves the affine scheme $Z$ above, or, in other words, that $g$ preserves the ideal $(P_1, \cdots, P_r)$. We may then consider the scheme $\Jtw Z$ of $g$--twisted jets to $Z$, where 
\[
\Jtw Z :=  \{ x(t^{1/m}) \in \on{Hom}(\spec \C[[t^{1/m}]], Z)  \vert x( \xim t^{1/m}) = g(x(t^{1/m}) ) \}
\]
We can write $\Jtw Z = \spec A^{\sigma}_{\infty}$,  where 
\begin{equation} \label{funct_tw_jet_sch}
A^{g}_{\infty} = \C [x_{1,n}, x_{2,n}, \cdots, x_{k,n}]_{} / (P^{g}_{1,n}, \cdots, P^{g}_{k,n}) \; \; \;   n \in \frac{\alpha_i}{m}+\mathbb{Z}, n \leq 0.
\end{equation}
and $P^{g}_{i,n}$ is the coefficient of $t^{n}$ in $P_i (x_{1}(t), x_2 (t), \cdots, x_k (t)) $

\section{Twisted modules from twisted jets} \label{twisted_mods}

Let $Z \subset \Ak$ be an affine scheme. The algebra $A_{\infty} = \mathbb{C}[\J Z]$ from (\ref{funct_jet_sch}) is equipped with a derivation $T$ defined by 
\[
T\cdot x_{i,n} = - (n-1) x_{i,n-1}
\]
We can write the polynomials $P_{i,n}$ in \ref{Pin} as
$$
P_{i,n} = \frac{T^{n}}{n!} P_{i,0}
$$
and thereby write
\begin{equation}
A_{\infty} = \mathbb{C}[\J Z] = \mathbb{C}[\J \Ak]/ (T^n P_{i,0}), \; \; 1 \leq i \leq r, n \geq 0.
\end{equation}
As explained in Example \ref{comm_va} $(A_{\infty}, T)$ carries a commutative vertex algebra structure
\begin{align*}
Y: A_{\infty} & \rightarrow \on{End}(A_{\infty})[[z,z^{-1}]] \\
Y(a,z) & = e^{zT} (a) =  \sum_{n \geq 0 } \frac{T^n(a)}{n!} z^{n}
\end{align*}

Let $g$ be an automorphism of $\Ak$ of finite order $m$ inducing an automorphism of $Z \subset \Ak$. $g$ acts on $\J Z$ by sending $x(t) \in \J Z$ to  $gx(t) := g(x(t))$, inducing an algebra automorphism $\ss$ of $A_{\infty}$ determined by $\ss(x_{i,n}) := g(x_{i})_n $. After diagonalizing $g$ as in \ref{sigma_action}, this action is given by $\ss (x_{i,n}) = \xim^{\alpha_i} x_{i,n}$. $\ss$ commutes with $T$, inducing a vertex algebra automorphism of $A_{\infty}$. Let $A^{g}_{\infty} = \mathbb{C}[\J^{g} Z] $ as in \ref{funct_tw_jet_sch}. 

\begin{theorem} \label{twisted_theorem}
$A^{g}_{\infty}$ carries the structure of a $\ss$--twisted $A_{\infty}$ vertex algebra module given by the assignment
\begin{equation}
Y_{g}: A_{\infty} \rightarrow \on{End}A^{g}_{\infty} [[ z^{1/m}, z^{-1/m} ]]
\end{equation} where
\begin{equation} \label{field_1}
Y_g (x_{i,0},z^{1/m})  = \sum_{ n \in \frac{\alpha_i}{m}+\mathbb{Z}, n \leq 0  } x_{i,n} z^{-n} \\
\end{equation} and
\begin{equation} \label{field_2}
Y_g (x_{i_1,n_1} \cdots x_{i_s,n_s}, z^{1/m}) := \prod^{s}_{j=1} \partial^{-n_j}_{z} Y(x_{i_j,0}, z^{1/m})
\end{equation}
\end{theorem}

\begin{proof} 
It follows from a twisted version of the reconstruction theorem for vertex operators in \cite{L} that $\C[\J^g \Ak ]$ has the structure of $\ss$-twisted $\C[\J \Ak]$--module with the field assignment \ref{field_1}. It then follows from the multiplicativity of $Y_g$ (see Remark \ref{multiplicativity} ) and the property \ref{twisted_derivative} that 
$Y_g$ is defined by \ref{field_2} on general elements of $\mathbb{C}[\J \Ak]$. It remains to show that the twisted module structure descends to the quotients $A_{\infty}, A^g_{\infty}$ . We have
\begin{align*}
Y_g(P_{i,n}, z^{1/m})= Y_g(\frac{T^{n}}{n!} P_{i,0}, z^{1/m}) & =  \frac{\partial^{n}_z}{n!} Y_g(P_{i,0},z^{1/m}) \\ &= \frac{\partial^{n}_z}{n!} P_i(x_1 (z^{1/m}), ... , x_k (z^{1/m}))
\end{align*}
It follows that if $P$ lies in the ideal generated by the $P_{i,n}$, then the coefficients of the field $Y_g(P,z^{1/m})$ lie in the ideal generated by the $P^{g}_{j,m}$, hence that $Y_g$ induces a well-defined structure on $A^{g}_{\infty}$ as a $\ss$--twisted $A_{\infty}$--module.  
\end{proof}

\subsection{Quasi-conformal structure} \label{quasiconf_structure}

We recall the quasi-conformal structure on the vertex algebra $A_{\infty}$ following \cite{FBZ}. It will be used to define coordinate-independent versions of twisted vertex operators in the next section.  
Let $\AutO$ denote the group of algebra automorphisms of $\ps$.  An automorphism $\rho \in \AutO$ is determined by where it sends the generator $z$, and can be written as
\[
\rho(z) = a_1 z + a_2 z^2 + \cdots
\]
where $a_1 \neq 0$.   We think of $\AutO$ as the automorphism group of the formal disk $D=\on{spec} \ps$ preserving the origin. The Lie algebra of $\AutO$ is spanned by the vector fields $L_n=-z^{n+1} \partial_z, \; n \geq 0$, satisfying the commutation relations
\[
[L_m, L_n] = (n-m) L_{n+m}
\]
The action of $\AutO$ on $\spec \ps$ induces an action on the jet scheme $\J Z$ and hence on $A_{\infty} = \mathbb{C}[\J Z]$.  The action of $Lie(\AutO)$ may be written explicitly as
\[
L_m \rightarrow \sum^k_{i=1} \sum_{n < 0}  -n x_{i,n}  \frac{\partial}{\partial x_{i,n-m}}
\]
Consider now the map
\[
f_m : \spec \psm \rightarrow \spec \ps
\]
induced by the inclusion $\ps \subset \psm$. We think of $\spec \psm$ as an $m$-th order ramified cover of $\spec \ps$. Let  $\AutmO$ be the group of automorphisms of $\psm$ preserving the subring $\ps$.  $\psi \in \AutmO$ may be written as
\[
\psi(z^{1/m}) = \sum_{n \geq 0} a_{\frac{1}{m}+n} z^{\frac{1}{m}+n} 
\]
where $a_{\frac{1}{m} \neq 0}$. $\AutmO$ can be thought of as the group of $m$-th roots of coordinate changes on $\spec \ps$. There is a short exact sequence of groups
\[
1 \rightarrow \mathbb{Z}/m \mathbb{Z} \rightarrow \AutmO \overset{h_m}{\rightarrow} \AutO \rightarrow 1
\]
where $h_m (\psi)(z)  = (\psi(z^{1/m}))^m$. The Lie algebra of $\AutmO$ is spanned by the vector fields $\wt{L}_n = -z^{1/m + n} \partial_{z^{1/m}}, \; n \geq 0$, and $h_m$ induces a Lie algebra isomomorphism
\begin{align*}
h_{m,*}: Lie(\AutmO) & \rightarrow Lie(\AutO) \\
\wt{L}_n &\rightarrow  m L_n
\end{align*} 
The action of $\AutmO$ on $\spec \psm$ induces an action on $\Jtw Z$ and hence on $A^g_{\infty}$. The action of $Lie(\AutmO)$ may be written explicitly as
\[
\wt{L}_r \rightarrow \sum^k_{i=1} \sum_{n \in \alpha_i/m + \mathbb{Z} ,n < 0} -nm x_{i,n}  \frac{\partial}{\partial x_{i,n-r}}
\]

\section{Coinvariants and conformal blocks} \label{coinvariants}

In this section, we study the spaces of coinvariants and conformal blocks for the vertex algebra $A_{\infty}$ on a stacky curve (i.e. orbicurve) with values in twisted modules of type $A^{g}_{\infty}$. We begin by briefly recalling the definition and construction of these spaces following the approach in \cite{FSz}, where we refer the reader for details.

Let $Y$ be a smooth complex projective curve. Given a point $y \in Y$, we denote by $\mc{O}_y$ the local ring at $y$, $\widehat{\mc{O}}_y$ its completion, and $\widehat{m}_y \subset \widehat{\mc{O}}_y$ the maximal ideal. We refer to a generator of $\widehat{m}_y$ as a \emph{formal coordinate} at $y$. Given such a formal coordinate $z$ at $y$, we may canonically identify $\widehat{\mc{O}}_y$ with $\mathbb{C}[[z]]$, and $\widehat{m}_y$ with $z \mathbb{C}[[z]]$. Let 
\[
\widehat{Y} = \{ (y, t_y) \vert y \in Y, t_y \in  \widehat{\mathcal{O}}_y,  (t_y) = \widehat{m}_y \}
\]
be the bundle of formal coordinates over $Y$. $\widehat{Y}$ is an $\AutO$--principal bundle, where the latter acts by changes of formal coordinates. $\AutO$ acts on $A_{\infty}$ as explained in  section \ref{quasiconf_structure}, and we may therefore form the associated bundle
\begin{equation} \label{ass_bundle}
\mc{A} = \widehat{Y} \underset{\AutO}{\times} A_{\infty}
\end{equation}

Given an open subset $U \subset Y$ equipped with an etale coordinate $f: U \rightarrow \mathbb{C}$, we obtain a section, hence trviailization
\begin{align*}
j_f : U & \rightarrow \widehat{Y}\vert_{U} \\
y & \rightarrow (y,f-f(y))
\end{align*}
This induces a trivialization of $\mc{A}$
\begin{align*}
\wt{j}_f: U \times A_{\infty} & \rightarrow \mathcal{A} \\
(y,p(x_{i,n})) & \rightarrow [y,f-f(y), p(x_{i,n}) ]
\end{align*}
where the notation $[y,f-f(y), p(x_{i,n}) ]$ is used to denote the equivalence class of element of $\mc{A}$ whose $\widehat{Y}$ component is $(y,f-f(y))$ and whose $A_{\infty}$ component is $p(x_{i,n})$. 

$\mc{A}$ is a bundle of commutative algebras over $Y$, and denoting by $\mc{A}_y$ the fiber of $\mc{A}$ at $y \in Y$, and $\spec \widehat{\mc{O}}_y$ by $\D_y$ we have a canonical identification 
\[
\spec \mc{A}_y = \on{Hom}(\D_y, Z)
\]
The relative spectrum $\spec \mathcal{A} \rightarrow Y$ is therefore identified with the bundle of jets of sections of $Y \times Z \rightarrow Y$, which we denote $\J_{Z \times Y/Y}$. $\mc{A}$ carries a canonical flat connection, which in a local coordinate $z$ is given by $\nabla_{\partial_z} = \partial_z - T$. Horizontal sections of $\mc{A}$ over $U \subset Y$ are holomorphic maps $U \rightarrow Z$, and since $Y$ is projective and $Z$ affine, global horizontal sections are constant maps to $Z$. 

\begin{rmk} \label{subring_section}
Note that the subring $A \subset A_{\infty}$ generated by the images of $x_{i,0}$ forms a trivial $\AutO$ sub-representation of $A_{\infty}$. Forming the associated bundle \ref{ass_bundle} and taking relative spectra, we obtain the canonical map $$ \mu: \J_{Z \times Y/Y} \rightarrow Z $$ which associates to the jet of a section $\phi \in \on{Hom}(\D_y, Z)$ its value  $\phi(y)$. 
\end{rmk}

Suppose now that $G$ is a finite group acting effectively on $Y$, and that $G$ acts on $\Ak$ preserving $Z \subset \Ak$. We may then define an action of $G$ on $\mc{A}$ by 
\[
g \cdot [y,t_y,p(x_{i,n})] = [gy, t_y \circ g^{-1}, p(x_{i,n} \circ g^{-1})] 
\]
which commutes with the flat connection $\nabla$. This $G$--equivariant structure on $\mc{A}$ allows us to descend it to a sheaf of algebras $\mc{A}_G$ over the orbicurve (or stacky curve) $[Y/G]$ equipped with a flat connection $\nabla^G$.  

Denote by 
\[
\pi: Y \rightarrow [Y/G]
\]
the projection to the quotient. Points of $[Y/G]$ will be denoted by $\wt{y}_1, \wt{y}_2, \cdots$. Given $\wt{y} \in [Y/G]$, we may write $$ \wt{y} = [y/G_y], $$ where $y \in \pi^{-1}(\wt{y})$, and $G_y$ is the stabilizer of $y$ in $G$. A point $\wt{y}$ with non-trivial $G_y$ is called a \emph{stacky point}.  It is well-known that $G_y$ cyclic, and we may choose a generator $g \in G_y$ and an etale coordinate $z^{1/m}$ in a neighborhood $U$ centered at $y$, with $m=\vert G_y \vert$, such that $z^{1/m} \circ g^{-1} = \xim z^{1/m}$. We call a coordinate possessing this property \emph{special}. The formal neighborhood of $\wt{y}$ in $[Y/G]$ can be described as the stack
\[
[\D_y/ G_y] = [\D_y / \langle g \rangle]
\]

The twisted module structure on $A^g_{\infty}, \, g \in G$ may be used to construct local sections of $\mc{A}^*_G$ on $[Y/G]$ as follows. Let
\[
\J_{\wt{y}} Z :=  \on{Hom}( [\D_y/ G_y], [Z/G_y] )
\]
$\J_{\wt{y}} Z$ is isomorphic to 
\[
\spec \C[[t^{1/m}]] \underset{\AutmO}{\times} \J^{g} Z
\]
and the choice of special formal coordinate $z^{1/m}$ allows us to identify $\J_{\wt{y}} Z$ with the twisted jet scheme $\J^{g} Z$. We thus obtain an isomorphism 
$$\lambda_{z^{1/m}}: \mc{A}_{\wt{y}}:= \mathbb{C}[\J_{\wt{y}} Z] \rightarrow A^g_{\infty}.$$ Denote by 
\[
 Y_g: A_{\infty} \rightarrow \on{End} A^g_{\infty} [[z^{1/m}, z^{-1/m}]]
\] 
the operator associated with the twisted vertex algebra module structure from Theorem \ref{twisted_theorem}, $\mc{K}_y$ the field of fractions of $ \widehat{\mathcal{O}}_y$, and $\D^{\times}_y = \spec \mc{K}_y $. We have $\mc{K}_y \simeq \mathbb{C}((z^{1/m}))$.  

We proceed to define a section 
\[
 \mc{Y}_{\wt{y}} \in \Gamma([\D^{\times}_y / G_y], \mc{A}_{G}^* \otimes \on{End}(\mc{A}_{\wt{y}})).
\]
Note that $z=(z^{1/m})^m$ is an etale coordinate on $U \backslash y$, and therefore yields a trivialization 
\[
\wt{j}_z : (U \backslash y) \times A_{\infty} \rightarrow \mc{A} \vert_{U \backslash y}
\]
which may be pulled back to $\D^{\times}_y$.  We define $\mc{Y}_{\wt{y}}$ by the property that 
\begin{equation}
 \langle \lambda^{t}_{z^{1/m}} (\phi), \mc{Y}_{\wt{y}}(\wt{j}_z (p)) \cdot \lambda^{-1}_{z^{1/m}}(q) \rangle := \langle \phi, Y_g(p,z^{1/m}) \cdot q \rangle \in \mathbb{C}((z^{1/m}))
\end{equation}
where $\phi \in (A^{g}_{\infty})^*, p \in A_{\infty}, q \in A^{g}_{\infty}$. The following is an immediate consequence of Theorem $5.1$ of \cite{FSz}, and provides a coordinate-independent description of the twisted vertex operation $Y_g$.

\begin{theorem}
The section $\mc{Y}_{\wt{y}}$ is independent of the choice of special coordinate $z^{1/m}$ and point $y \in \pi^{-1}(\wt{y})$, and thus defines a canonical element 
\[
 \mc{Y}_{\wt{y}} \in \Gamma([\D^{\times}_y / G_y], \mc{A}_{G}^* \otimes \on{End}(\mc{A}_{\wt{y}})).
\]
\end{theorem}

\begin{rmk}
We think of the stack $[\D^{\times}_y / G_y]$ as a "punctured formal disk" around $\wt{y} \in [Y/G]$. 
\end{rmk}

We now define the spaces of coinvariants  and conformal blocks for the commutative vertex algebra $A_{\infty}$ on $[Y/G]$. Let $\{ \wt{y}_1, \cdots, \wt{y}_s \}$ be a non-empty set of points of $[Y/G]$ which includes all the stacky points. 
Let $$Y^{\circ} = Y\backslash \pi^{-1}(\wt{y}_1) \cup \cdots \cup \pi^{-1}(\wt{y}_s).$$ We have $$[Y^{\circ} / G] = Y^{\circ}/G =  [Y/G] \backslash \{\wt{y}_1, \cdots, \wt{y}_s\}.$$ Let $$ \mc{A}_{out} = \Gamma([Y^{\circ}/G], \mc{A}_G \otimes \Omega^{1} ), $$  or equivalently, the $G$--invariant sections of $\mc{A} \otimes \Omega^1$ over $Y^{\circ}$. For each $\wt{y}_j$, there is a map
\[
\alpha_j: \mc{A}_{out} \rightarrow \on{End} (\mc{A}_{\wt{y}_j})
\]
given by 
\[
\omega \in \mc{A}_{out} \rightarrow Res_{\wt{y}_j} \langle \omega, \mc{Y}_{\wt{y}_j} \rangle \in \on{End} (\mc{A}_{\wt{y}_j})
\]
$\omega \in \mc{A}_{out}$ thus acts on $\mc{A}_{\wt{y}_1} \otimes \cdots \otimes \mc{A}_{\wt{y}_s}$ by
$$
s \cdot (a_1 \otimes \cdots \otimes a_s) := \sum^{s}_{j=1} a_1 \otimes \cdots \otimes \alpha_{j}(s) \cdot a_j \otimes \cdots \otimes a_s
$$

\begin{definition} 
Let $G$, $Y$, $Z$, $\wt{y}_1, \cdots, \wt{y}_s \in [Y/G]$ be as above.
The \emph{space of coinvariants} for the vertex algebra $\J Z$ on the orbicurve $[Y/G]$ is 
$$
      \coinv := \mc{A}_{\wt{y}_1} \otimes \cdots \otimes \mc{A}_{\wt{y}_s} / \mc{A}_{out} \cdot (\mc{A}_{\wt{y}_1} \otimes \cdots \otimes A_{\wt{y}_s})
$$
where $  \mc{A}_{out} \cdot (\mc{A}_{\wt{y}_1} \otimes \cdots \otimes A_{\wt{y}_s}) $ denotes the ideal generated by $\mc{A}_{out}$ in the algebra $  \mc{A}_{\wt{y}_1} \otimes \cdots \otimes A_{\wt{y}_s}). $

 The dual space
\[
\cblocks := \on{Hom}_{\mathbb{C}} (\coinv, \mathbb{C})
\]
is called the \emph{space of conformal blocks} for the vertex algebra $\J Z$ on the orbicurve $[Y/G]$.
\end{definition}

\begin{rmk} \label{coinv_is_algebra}
 $\coinv$, being the quotient of a commutative algebra by an ideal, has the structure of a $\mathbb{C}$-algebra. 
\end{rmk}

%

The definition of $\cblocks$ above is given in terms of the of the orbicurve $[Y/G]$ and the sheaf $\mc{A}_G$ on $[Y/G]$. As explained in \cite{FSz}, by using the strong residue theorem,  we may restate the definition in terms of curve $Y$ and the $G$--equivariant structure on $\mc{A}$ as follows: 

\begin{definition}
Let $G$, $Y$, $Z$, $\wt{y}_1, \cdots, \wt{y}_s \in [Y/G]$ be as above. Choose $y_j \in \pi^{-1}(\wt{y}_j)$ for $j=1, \cdots, s$.  
The space of \emph{conformal blocks} for the vertex algebra $\J Z$ on the orbicurve $[Y/G]$ is 
$$
\cblocks :=  \varphi \in (A_{\wt{y}_1} \otimes \cdots A_{\wt{y}_s})^* $$  such that  $$ \forall \, a_1 \in A_{\wt{y}_1}, \cdots, a_s \in A_{\wt{y}_s} $$
the sections 
\begin{equation}
\langle \varphi, a_1 \otimes \cdots \otimes \mc{Y}_{\wt{y}_j} \cdot a_j \otimes \cdots \otimes a_s \rangle \; \in \Gamma(\D^{\times}_{y_j}, \mc{A}^*) \; j=1, \cdots, s
\end{equation}
extend to  a single $G$--invariant horizontal section $\mc{Y}_{\varphi}$ of $\mc{A}^{*}$ on $Y^{\circ} $.
\end{definition}

\begin{rmk}
As explained in \cite{FSz}, the space $\cblocks$ is independent of the choice of $y_j \in \pi^{-1}(\wt{y}_j)$. This choice is made only to provide a concrete model of the formal punctured neighborhood of $\wt{y}_j \in [Y/G]$ as $[\D^{\times}_{y_j}/G_{y_j}]$ 
\end{rmk}



We now state our main result regarding the spaces $\coinv$. Denote by $Z^G \subset Z$ the closed sub-scheme of $G$--fixed points, and $\mathbb{C}[Z^G]$ its coordinate ring. 

\begin{theorem} \label{coinv_theorem}
The space of coinvariants $\coinv$ is isomorphic to $\mathbb{C}[Z^G]$.
\end{theorem}
\begin{proof}
By remark \ref{coinv_is_algebra}, $\coinv$ has the structure of commutative $\mathbb{C}$--algebra. Let $$ \varphi: \coinv \rightarrow \mathbb{C} $$ be a closed point of $\spec \coinv$. Then in particular, $\varphi \in \cblocks$, and we denote by $\omega_{\varphi} \in \Gamma(Y^{\circ}, \mc{A}^*) $ the $G$--invariant horizontal section of $\mc{A}^*$ which near $y_j \in \pi^{-1}(\wt{y}_j)$ agrees with 
$$
\langle \varphi, 1 \otimes \cdots \otimes \mc{Y}_{\wt{y}_j}\cdot 1 \otimes \cdots \otimes 1 \rangle \; \in \Gamma(\D^{\times}_{y_j}, \mc{A}^*) 
$$
for $ j=1, \cdots, s$. As explained in Proposition  of \cite{FBZ}, for each $y \in Y^{\circ}$
the restriction
$$ \omega_{\varphi}: \mc{A}_y \rightarrow \mathbb{C} $$
is a ring homomorphism, whose associated point of the jet scheme  $\on{Hom}(\D_y, Z)$ is the jet at $y$ of a map 
$$h_{\varphi}: Y^{\circ} \rightarrow Z . $$ 
The $G$--invariance of $\omega_{\varphi}$ ensures that $h_{\varphi}$ is $G$--invariant as well, i.e. satisfies $h_{\varphi}(g \cdot y) = g \cdot h_{\varphi}(y)$. We now show that $h_{\varphi}$ extends to all of $Y$. It follows from Remark \ref{subring_section}  that the composition 
\[
\mathbb{C}[Z] \rightarrow \mc{A}_y \overset{\omega_{\varphi}}{\rightarrow} \mathbb{C} 
\]
is simply the map 
\[
p \rightarrow p(h_{\varphi}(y))
\]
If $y_j \in \pi^{-1}(\wt{y}_j)$, then after fixing a special coordinate $z^{1/m}$ near $y_j$, we have that on $\D^{\times}_{y_j}$, 
\begin{equation} \label{ev_section}
\omega_{\varphi}( \wt{j}_z (p) ) = \langle \varphi, Y(p,z^{1/m}) \cdot 1 \rangle \in \mathbb{C}(( z^{1/m} ))
\end{equation}
However, note that for $p \in \mathbb{C}[Z] \subset A_{\infty}$, $Y(p, z^{1/m}) \in \mathbb{C}[[z^{1/m}]]$.  Thus, the limit $z^{1/m} \rightarrow 0$ is well-defined, showing that $p(h_{\varphi}(y))$ is well-defined, hence that $h_{\varphi}$ extends to $y_j$. Since $Y$ is projective, and $Z$ affine, $h_{\varphi}$ is constant. The $G$--invariance forces the image to lie in $Z^{G}$. 

\end{proof}

\begin{rmk}
When $G$ is trivial, and $s=1$, we recover a result proven in section $9.4.4$ of \cite{FBZ} identifying the space of one-point coinvariants of the vertex algebra $A_{\infty}$ on $Y$ with $\C[Z]$. 
\end{rmk}

\address{\tiny DEPARTMENT OF MATHEMATICS AND STATISTICS, BOSTON UNIVERSITY, 111 CUMMINGTON MALL, BOSTON} \\
\indent \footnotesize{\email{szczesny@math.bu.edu}}


\newpage
\begin{thebibliography}{99}

\bibitem{A} Arakawa, T.  A remark on the C2-cofiniteness condition on vertex algebras. Math. Z. 270 (2012), no. 1-2, 559?575.
\bibitem{AM} Arakawa T.; Moreau A. Sheets and associated varieties of vertex algebras. Preprint {\tt arXiv:1601.05906}
\bibitem{BD} Beilinson A.; Drinfeld V. Chiral Algebras.  American Mathematical Society Colloquium Publications, 51. American Mathematical Society, Providence, RI, 2004.
\bibitem{FBZ} Frenkel E.; Ben-Zvi David. Vertex algebras and algebraic curves. Second edition. Mathematical Surveys and Monographs, 88. American Mathematical Society, Providence, RI, 2004.
\bibitem{FSz} Frenkel E.; Szczesny M.  Twisted modules over vertex algebras on algebraic curves, Adv. Math. {\tt 187} no. 1 (2004), 195-227.
\bibitem{K} Kac, V. Vertex algebras for beginners (2nd ed.). AMS - University Lecture Series. 1998.
\bibitem{L} Li, H.  Local systems of twisted vertex operators, vertex operator superalgebras and twisted modules.  Moonshine, the Monster, and related topics (South Hadley, MA, 1994), 203?236, Contemp. Math., 193, Amer. Math. Soc., Providence, RI, 1996. 
\bibitem{LSS} Linshaw, Andrew R.; Schwarz, Gerald W.; Song, Bailin Arc spaces and the vertex algebra commutant problem. Adv. Math. 277 (2015), 338?364.


\end{thebibliography}
\end{document}